\documentclass{amsart}

\usepackage{amsmath,amsfonts,amsthm,amssymb,verbatim}

\newtheorem{lemma}[equation]{Lemma}
\newtheorem{proposition}[equation]{Proposition}
\newtheorem{corollary}[equation]{Corollary}
\newtheorem{theorem}[equation]{Theorem}

\numberwithin{equation}{section}

\begin{document}

\title[On unit root formulas for toric exponential sums]{On unit root formulas 
for\\  toric exponential sums}
\author{Alan Adolphson}
\address{Department of Mathematics\\ Oklahoma State University\\ Stillwater, OK 
74078}
\email{adolphs@math.okstate.edu}
\author{Steven Sperber}
\address{School of Mathematics\\ University of Minnesota\\ Minneapolis, MN 5545}
\email{sperber@math.umn.edu}
\date{\today}
\begin{abstract}
Starting from a classical generating series for Bessel functions due to 
Schl\"omilch\cite{Sc}, we use Dwork's relative dual theory to broadly generalize unit-root results of Dwork\cite{D2} on Kloosterman sums and Sperber\cite{Sp} on hyperkloosterman sums.  In particular,
we express the (unique) $p$-adic unit root of an arbitrary exponential sum on the torus ${\bf T}^n$ in terms of special values of the $p$-adic analytic continuation of a ratio of $A$-hypergeometric functions.  In contrast with the earlier works, we use noncohomological methods and obtain results that are valid for arbitrary exponential sums without any hypothesis of nondegeneracy.
\end{abstract}
\maketitle

\section{Introduction}

The starting point for this work is the classical generating series
\[ \exp\frac{1}{2}(\Lambda X - \Lambda/X) = \sum_{i\in{\bf Z}} J_i(\Lambda)X^i \]
for the Bessel functions $\{J_i(\Lambda)\}_{i\in{\bf Z}}$ due to Schl\"omilch\cite{Sc} that was the foundation for his treatment of Bessel functions (see \cite[page~14]{W}).  Suitably normalized, it also played a fundamental role in Dwork's construction\cite{D2} of $p$-adic cohomology for~$J_0(\Lambda)$.  Our realization that the series itself (suitably normalized) could be viewed as a distinguished element in Dwork's relative dual complex led us to the present generalization.

Let $A\subseteq{\bf Z}^n$ be a finite subset that spans ${\bf R}^n$ as
real vector space and set
\[ f_\Lambda(X) = \sum_{a\in A} \Lambda_aX^a\in{\bf Z}[\{\Lambda_a\}_{a\in
A}][X_1^{\pm 1},\dots,X_n^{\pm 1}], \]
where the $\Lambda_a$ and the $X_i$ are indeterminates and where $X^a =
X_1^{a_1}\cdots X_n^{a_n}$ for $a = (a_1,\dots,a_n)$. Let ${\bf F}_q$ be the finite field of $q=p^\epsilon$ elements, $p$ a prime, and let $\bar{\bf F}_q$ be its algebraic closure. For each $\bar{\lambda} =
(\bar{\lambda}_a)_{a\in A}\in(\bar{\bf F}_q)^{|A|}$, let
\[ f_{\bar\lambda}(X) = \sum_{a\in A} \bar{\lambda}_a X^a\in{\bf
F}_q(\bar{\lambda})[X_1^{\pm 1},\dots,X_n^{\pm 1}], \]
a regular function on the $n$-torus ${\bf T}^n$ over ${\bf F}_q(\bar{\lambda})$.
Fix a nontrivial additive character $\Theta:{\bf F}_q\to {\bf Q}_p(\zeta_p)$ and
let $\Theta_{\bar{\lambda}}$ be the additive character $\Theta_{\bar{\lambda}}=
\Theta\circ{\rm Tr}_{{\bf F}_q(\bar{\lambda})/{\bf F}_q}$ of the field ${\bf
F}_q(\bar{\lambda})$. For each positive integer $l$, let ${\bf
F}_q(\bar{\lambda},l)$ denote the extension of degree~$l$ of ${\bf
F}_q(\bar{\lambda})$ and define an exponential sum
\[ S_l = S_l(f_{\bar{\lambda}},\Theta_{\bar{\lambda}},{\bf T}^n) = \sum_{x\in {\bf
T}^n({\bf F}_q(\bar{\lambda},l))} \Theta_{\bar{\lambda}}\circ{\rm Tr}_{{\bf
F}_q(\bar{\lambda},l)/{\bf F}_q(\bar{\lambda})}(f_{\bar{\lambda}}(x)). \]
The associated $L$-function is
\[ L(f_{\bar{\lambda}};T) = L(f_{\bar{\lambda}},\Theta_{\bar{\lambda}},{\bf T}^n;T)=
\exp\biggl(\sum_{l=1}^\infty S_l\frac{T^l}{l}\bigg). \] 

It is well-known that $L(f_{\bar{\lambda}};T)\in{\bf Q}(\zeta_p)(T)$ and that its reciprocal zeros and poles are algebraic integers. We note that among these
reciprocal zeros and poles there must be at least one $p$-adic unit: if ${\bf F}_q(\bar{\lambda})$ has cardinality $q^\kappa$, then $S_l$ is the sum
of $(q^{\kappa l}-1)^n$ $p$-th roots of unity, so $S_l$ itself is a $p$-adic
unit for every $l$. On the other hand, a simple consequence of the
Dwork trace formula will imply (see Section~3) that there is at most a
single unit root, and it must occur amongst the reciprocal zeros (as
opposed to the reciprocal poles) of
$L(f_{\bar{\lambda}};T)^{(-1)^{n+1}}$. We denote this unit root by
$u(\bar{\lambda})$. It is the goal of this work to exhibit an explicit
$p$-adic analytic formula for $u(\bar{\lambda})$ in terms of certain
$A$-hypergeometric functions.

Consider the series
\begin{align}
\exp f_\Lambda(X) &= \prod_{a\in A} \exp(\Lambda_aX^a) \\
 & = \sum_{i\in {\bf Z}^n} F_i(\Lambda) X^i \nonumber
\end{align}
where the $F_i(\Lambda)$ lie in ${\bf Q}[[\Lambda]]$.  Explicitly, one has
\begin{equation}
F_i(\Lambda) = \sum_{\substack{u = (u_a)_{a\in A} \\ \sum_{a\in A} u_a a = i}}
\frac{\Lambda^u}{\prod_{a\in A} (u_a!)}.
\end{equation}

The $A$-hypergeometric system with parameter $\alpha =
(\alpha_1,\dots,\alpha_n)\in{\bf C}^n$ (where ${\bf C}$ denotes the complex numbers) is the system of partial differential
equations consisting of the operators
\[ \Box_\ell = \prod_{\ell_a>0}\biggl(\frac{\partial}{\partial
\Lambda_a}\biggr)^{\ell_a} - \prod_{\ell_a<0} \biggl(\frac{\partial}{\partial
\Lambda_a}\biggr)^{-\ell_a} \]
for all $\ell = (\ell_a)_{a\in A}\in{\bf Z}^{|A|}$ satisfying $\sum_{a\in A}\ell_a
a = 0$
and the operators
\[ Z_j = \sum_{a\in A} a_j\Lambda_a\frac{\partial}{\partial \Lambda_a} - \alpha_j
\]
for $a=(a_1,\dots,a_n)\in A$ and $j=1,\dots,n$. Using Equations (1.1) and (1.2), it is straightforward to check that for $i\in{\bf Z}^n$, $F_i(\Lambda)$ satisfies the
$A$-hypergeometric system with parameter $i$.

Fix $\pi$ satisfying $\pi^{p-1} = -p$ and $\Theta(1) \equiv \pi\pmod{\pi^2}$. It
follows from Equation~(1.2) that the $F_i(\pi\Lambda)$ converge $p$-adically for all
$\Lambda$ satisfying $|\Lambda_a|<1$ for all $a\in A$. Let ${\mathcal F}(\Lambda)=F_0(\pi\Lambda)/F_0(\pi\Lambda^p)$.  The main result of this paper is the following statement.  Note that we make no restriction (such as nondegeneracy) on the choice of $\bar{\lambda}\in(\bar{\bf F}_q)^{|A|}$.  

 \begin{theorem}
The series ${\mathcal F}(\Lambda)$ converges $p$-adically for $|\Lambda_a|\leq 1$ for all $a\in A$ and the unit root of $L(f_{\bar{\lambda}};T)$ is given by
\[ u(\bar{\lambda}) = {\mathcal F}({\lambda}){\mathcal F}({\lambda}^p){\mathcal F}(\lambda^{p^2})\cdots
{\mathcal F}({\lambda}^{p^{\epsilon d(\bar{\lambda})-1}}), \]
where ${\lambda}$ denotes the Teichm\"{u}ller lifting of $\bar{\lambda}$ and $d(\bar{\lambda}) = [{\bf F}_q(\bar{\lambda}):{\bf F}_q]$.
\end{theorem}

\section{Analytic continuation}

We begin by proving the analytic continuation of the function ${\mathcal F}$ defined in the introduction.

Let $C\subseteq{\bf R}^n$ be the real cone generated by the elements of $A$ and let $\Delta\subseteq{\bf R}^n$ be the convex hull of the set $A\cup\{(0,\dots,0)\}$.  
Put $M=C\cap{\bf Z}^n$. For $\nu\in M$, define the {\em weight}\/ of $\nu$,
$w(\nu)$, to be the least nonnegative real (hence rational) number such that
$\nu\in w(\nu)\Delta$. There exists $D\in{\bf Z}_{>0}$ such that $w(\nu)\in{\bf
Q}_{\geq 0}\cap{\bf Z}[1/D]$. The weight function $w$ is easily seen to have the
following properties:
\begin{align*}
    \text{(i) }& w(\nu) \geq 0 \text{ and } w(\nu) = 0 \text{ if and only
  if } \nu  = 0, \\
\text{(ii) }& w(c\nu) = cw(\nu) \text{ for } c \in {\bf Z}_{\geq 0}, \\
\text{(iii) }& w(\nu + \mu) \leq w(\nu) + w(\mu) \text{ with equality holding
  if and only if } \nu \text{ and } \mu \text{ are} \\
&\text{ cofacial, that is, } \nu \text{ and } \mu \text{ lie in a cone
  over the same closed face of } \Delta. \\
\text{(iv) }& \text{If $\dim\Delta=n$, let $\{\ell_i\}_{i=1}^N$ be linear forms such that the codimension-one faces}\\
 &\text{of $\Delta$ not containing the origin lie in the hyperplanes $\{\ell_i=1\}_{i=1}^N$.  Then}\\
& w(\nu) = \max\{\ell_i(\nu\}_{i=1}^N.
\end{align*}

Let $\Omega$ be a finite extension of ${\bf Q}_p$ containing $\pi$ and an element $\tilde{\pi}$ satisfying ${\rm ord}\:\tilde{\pi} = (p-1)/p^2$ (we always normalize the valuation so that ${\rm ord}\:p = 1$).  Put 
\[ R = \bigg\{ \xi(\Lambda) = \sum_{\nu\in({\bf Z}_{\geq 0})^{|A|}} c_\nu\Lambda^\nu\mid \text{$c_\nu\in\Omega$ and $\{|c_\nu|\}_\nu$ is bounded}\bigg\}, \]
\[ R' = \bigg\{ \xi(\Lambda) = \sum_{\nu\in({\bf Z}_{\geq 0})^{|A|}} c_\nu\Lambda^\nu\mid \text{$c_\nu\in\Omega$ and $c_\nu\to 0$ as $\nu\to\infty$}\bigg\}. \]
Equivalently, $R$ (resp.\ $R'$) is the ring of formal power series in $\{\Lambda_a\}_{a\in A}$ that converge on the open unit polydisk in $\Omega^{|A|}$ (resp.\ the closed unit polydisk in $\Omega^{|A|}$).  Define a norm on $R$ by setting $|\xi(\Lambda)| = \sup_\nu\{|c_\nu|\}$.  Both $R$ and $R'$ are complete in this norm.
Note that $(1.2)$ implies that the coefficients $F_i(\pi\Lambda)$ of $\exp\pi f_{\Lambda}(X)$ belong to~$R$.

Let $S$ be the set
\begin{multline*}
 S = \\
\bigg\{\xi(\Lambda,X) = \sum_{\mu\in M} \xi_\mu(\Lambda) \tilde{\pi}^{-w(\mu)}X^{-\mu} \mid \text{$\xi_\mu(\Lambda)\in R$ and $\{|\xi_\mu(\Lambda)|\}_\mu$ is bounded}\bigg\}. 
\end{multline*}
Let $S'$ be defined analogously with the conditions ``$\xi_\mu(\Lambda)\in R$'' replaced by ``$\xi_\mu(\Lambda)\in R'$''.  Define a norm on $S$ by setting
\[ |\xi(\Lambda,X)| = \sup_\mu\{|\xi_\mu(\Lambda)|\}. \]
Both $S$ and $S'$ are complete under this norm.

Define $\theta(t) = \exp(\pi(t-t^p)) = \sum_{i=0}^\infty b_it^i$.  One has (Dwork\cite[Section 4a)]{D})
\begin{equation}
{\rm ord}\: b_i\geq \frac{i(p-1)}{p^2}.
\end{equation}
Let
\[ F(\Lambda,X) = \prod_{a\in A}\theta(\Lambda_aX^a) = \sum_{\mu\in M} B_\mu(\Lambda)X^\mu. \]

\begin{lemma}
One has $B_\mu(\Lambda)\in R'$ and $|B_\mu(\Lambda)|\leq |\tilde{\pi}|^{w(\mu)}$.
\end{lemma}

\begin{proof}
From the definition,
\[ B_\mu(\Lambda) = \sum_{\nu\in({\bf Z}_{\geq 0})^{|A|}} B^{(\mu)}_\nu\Lambda^\nu, \]
where
\[ B^{(\mu)}_\nu = \begin{cases} \prod_{a\in A}b_{\nu_a}  & \text{if $\sum_{a\in A}\nu_a a = \mu$,} \\ 0 & \text{if $\sum_{a\in A} \nu_a a\neq\mu$.} \end{cases} \]
It follows from (2.1) that $B^{(\mu)}_\nu\to 0$ as $\nu\to\infty$, which shows that $B_\mu(\Lambda)\in R'$.  We have
\[ {\rm ord}\: B^{(\mu)}_\nu\geq\sum_{a\in A}{\rm ord}\: b_{\nu_a}\geq \sum_{a\in A}\frac{\nu_a(p-1)}{p^2}\geq w(\mu)\frac{p-1}{p^2}, \]
which implies $|B_\mu(\Lambda)|\leq |\tilde{\pi}|^{w(\mu)}$.
\end{proof}

By the proof of Lemma 2.2, we may write $B^{(\mu)}_\nu = \tilde{\pi}^{w(\mu)}\tilde{B}^{(\mu)}_\nu$ with $|\tilde{B}^{(\mu)}_\nu|\leq 1$.  We may then write $B_\mu(\Lambda) = \tilde{\pi}^{w(\mu)}\tilde{B}_\mu(\Lambda)$ with $\tilde{B}_\mu(\Lambda) = \sum_\nu \tilde{B}^{(\mu)}_\nu \Lambda^\nu$ and $|\tilde{B}_\mu(\Lambda)|\leq 1$.  Let
\[ \xi(\Lambda,X) = \sum_{\nu\in M} \xi_\nu(\Lambda)\tilde{\pi}^{-w(\nu)}X^{-\nu}\in S. \]
We claim that the product $F(\Lambda,X)\xi(\Lambda^p,X^p)$ is well-defined.  Formally we have
\[ F(\Lambda,X)\xi(\Lambda^p,X^p) = \sum_{\rho\in{\bf Z}^n} \zeta_\rho(\Lambda)X^{-\rho}, \]
where
\begin{equation}
\zeta_\rho(\Lambda) = \sum_{\substack{\mu,\nu\in M \\ \mu-p\nu = -\rho}} \tilde{\pi}^{w(\mu)-w(\nu)}\tilde{B}_\mu(\Lambda)\xi_\nu(\Lambda^p). 
\end{equation}
To prove convergence of this series, we need to show that $w(\mu)-w(\nu)\to\infty$ as $\nu\to\infty$.  By property~(iv) of the weight function, for a given $\nu\in M$ we may choose a linear form $\ell$ (depending on $\nu$) for which $w(\nu) = \ell(\nu)$ while $w(\mu)\geq\ell(\mu)$.  Since $\mu = p\nu-\rho$, we get
\begin{equation}
w(\mu)-w(\nu) \geq \ell(\mu-\nu) = \ell((p-1)\nu)-\ell(\rho) = (p-1)w(\nu)-\ell(\rho).
\end{equation}
As $\nu\to\infty$, $(p-1)w(\nu)\to\infty$ while $\ell(\rho)$ takes values in a finite set of rational numbers (there are only finitely many possibilities for $\ell$).  This gives the desired result.

For a formal series $\sum_{\rho\in{\bf Z}^n} \zeta_\rho(\Lambda)X^{-\rho}$ with $\zeta_\rho(\Lambda)\in\Omega[[\Lambda]]$, define
\[ \gamma'\bigg(\sum_{\rho\in{\bf Z}^n} \zeta_\rho(\Lambda) X^{-\rho}\bigg) = \sum_{\rho\in M}\zeta_\rho(\Lambda) X^{-\rho} \]
and define for $\xi(\Lambda,X)\in S$
\begin{align*}
\alpha^*(\xi(\Lambda,X)) &= \gamma'(F(\Lambda,X)\xi(\Lambda^p,X^p)) \\
 &= \sum_{\rho\in M}\zeta_\rho(\Lambda)X^{-\rho}. 
\end{align*}
For $\rho\in M$ put $\eta_\rho(\Lambda) = \tilde{\pi}^{w(\rho)}\zeta_\rho(\Lambda)$, so that
\begin{equation}
\alpha^*(\xi(\Lambda,X)) = \sum_{\rho\in M} \eta_\rho(\Lambda)\tilde{\pi}^{-w(\rho)} X^{-\rho} 
\end{equation}
with
\begin{equation}
\eta_\rho(\Lambda) = \sum_{\substack{\mu,\nu\in M\\ \mu-p\nu = \rho}} \tilde{\pi}^{w(\rho)+w(\mu)-w(\nu)}\tilde{B}_\mu(\Lambda)\xi_\nu(\Lambda^p). \end{equation}
Since $w(\rho)\geq\ell(\rho)$ for $\rho\in M$, Equation (2.4) implies that 
\begin{equation}
w(\rho)+w(\mu)-w(\nu)\geq (p-1)w(\nu), 
\end{equation}
so by Equation (2.6), $|\eta_\rho(\Lambda)|\leq |\xi(\Lambda,X)|$ for all $\rho\in M$. This shows $\alpha^*(\xi(\Lambda,X))\in S$ and
\[ |\alpha^*(\xi(\Lambda,X))|\leq |\xi(\Lambda,X)|. \]
Furthermore, this argument also shows that $\alpha^*(S')\subseteq S'$.

\begin{lemma}
If $|\xi_0(\Lambda)|\leq |\tilde{\pi}|^{(p-1)/D}$, then $|\alpha^*(\xi(\Lambda,X))|\leq |\tilde{\pi}|^{(p-1)/D}|\xi(\Lambda,X)|$.
\end{lemma}

\begin{proof}
This follows immediately from Equations (2.6) and (2.7) since $w(\nu)\geq 1/D$ for $\nu\neq 0$.
\end{proof}

From Equation (2.6), we have
\begin{equation}
\eta_0(\Lambda) = \sum_{\nu\in M} \tilde{B}_{p\nu}(\Lambda)\xi_\nu(\Lambda^p)\tilde{\pi}^{(p-1)w(\nu)}.
\end{equation}
Note that $\tilde{B}_0(\Lambda) = B_0(\Lambda)\equiv 1\pmod{\tilde{\pi}}$ since ${\rm ord}\:b_i>0$ for all $i>0$ implies ${\rm ord}\:B^{(0)}_\nu>0$ for all $\nu\neq 0$.  Thus $B_0(\Lambda)$ is an invertible element of $R'$.  The following lemma is then immediate from Equation (2.9).

\begin{lemma}
If $\xi_0(\Lambda)$ is an invertible element of $R$ (resp.\ $R'$), then so is $\eta_0(\Lambda)$.
\end{lemma}

Put 
\[ T = \{\xi(\Lambda,X)\in S\mid \text{$|\xi(\Lambda,X)|\leq 1$ and $\xi_0(\Lambda) = 1$}\} \]
and put $T' = T\cap S'$.  Using the notation of Equation~(2.5), define $\beta:T\to T$ by
\[ \beta(\xi(\Lambda,X)) = \frac{\alpha^*(\xi(\Lambda,X))}{\eta_0(\Lambda)}. \]
Note that $\beta(T')\subseteq T'$.

\begin{proposition}
The operator $\beta$ is a contraction mapping on the complete metric space $T$.  More precisely, if $\xi^{(1)}(\Lambda,X),\xi^{(2)}(\Lambda,X)\in T$, then
\[ |\beta(\xi^{(1)}(\Lambda,X))-\beta(\xi^{(2)}(\Lambda,X))|\leq |\tilde{\pi}|^{(p-1)/D}|\xi^{(1)}(\Lambda,X)-\xi^{(2)}(\Lambda,X)|. \]
\end{proposition}

\begin{proof}
We have (in the obvious notation)
\begin{equation*}
\begin{split}
\beta(\xi^{(1)}(\Lambda,X))-\beta(\xi^{(2)}(\Lambda,X)) &= \frac{\alpha^*(\xi^{(1)}(\Lambda,X))}{\eta^{(1)}_0(\Lambda)} - 
\frac{\alpha^*(\xi^{(2)}(\Lambda,X))}{\eta^{(2)}_0(\Lambda)} \\
 &= \frac{\alpha^*(\xi^{(1)}(\Lambda,X)-\xi^{(2)}(\Lambda,X))}{\eta^{(1)}_0(\Lambda)} \\
 & \qquad - \alpha^*(\xi^{(2)}(\Lambda,X))\frac{\eta^{(1)}_0(\Lambda) - \eta^{(2)}_0(\Lambda)}{\eta^{(1)}_0(\Lambda)\eta^{(2)}_0(\Lambda)}.
\end{split}
\end{equation*}
Since $\eta^{(1)}_0(\Lambda)-\eta^{(2)}_0(\Lambda)$ is the coefficient of $X^0$ in $\alpha^*(\xi^{(1)}(\Lambda,X)-\xi^{(2)}(\Lambda,X))$, we have
\[ |\eta^{(1)}_0(\Lambda)-\eta^{(2)}_0(\Lambda)|\leq
|\alpha^*(\xi^{(1)}(\Lambda,X)-\xi^{(2)}(\Lambda,X))|. \]
And since the coefficient of $X^0$ in $\xi^{(1)}(\Lambda,X)-\xi^{(2)}(\Lambda,X)$ equals $0$, the proposition follows from Lemma~2.8.
\end{proof}

{\it Remark}:  Proposition 2.11 implies that $\beta$ has a unique fixed point in $T$.  And since $\beta$ is stable on $T'$, that fixed point must lie in $T'$.  Let $\xi(\Lambda,X)\in T'$ be the unique fixed point of $\beta$.  The equation $\beta(\xi(\Lambda,X)) = \xi(\Lambda,X)$ is equivalent to the equation 
\[ \alpha^*(\xi(\Lambda,X)) = \eta_0(\Lambda)\xi(\Lambda,X). \]
Since $\alpha^*$ is stable on $S'$, it follows that
\begin{equation}
\eta_0(\Lambda)\xi_\mu(\Lambda)\in R'\quad\text{for all $\mu\in M$}. 
\end{equation}
In particular, since $\xi_0(\Lambda) = 1$, we have $\eta_0(\Lambda)\in R'$.

Put $C_0=C\cap(-C)$, the largest subspace of ${\bf R}^n$ contained in $C$, and put $M_0 = {\bf Z}^n\cap C_0$, a subgroup of $M$.   
For a formal series $\sum_{\mu\in{\bf Z}^n} c_\mu(\Lambda)X^\mu$ with $c_\mu(\Lambda)\in\Omega[[\Lambda]]$ we define
\[ \gamma\bigg(\sum_{\mu\in{\bf Z}^n} c_\mu(\Lambda)X^\mu\bigg) = \sum_{\mu\in M_0} c_\mu(\Lambda)X^\mu \]
and set
\[ \zeta(\Lambda,X) = \gamma(\exp(\pi f_\Lambda(X))). \]
Of course, when the origin is an interior point of $\Delta$, then $M_0={\bf Z}^n$ and $\zeta(\Lambda,X) = \exp(\pi f_\Lambda(X))$.  In any case, the coefficients of $\zeta(\Lambda,X)$ belong to $R$.

Since $\exp(\pi f_\Lambda(X)) = \prod_{a\in A}\exp(\pi\Lambda_aX^a)$, we can expand this product to get
\begin{align*}
\zeta(\Lambda,X) &= \gamma\bigg(\prod_{a\in A} \sum_{\nu_a=0}^\infty \frac{(\pi\Lambda_aX^a)^{\nu_a}}{\nu_a!}\bigg) \\
 &= \sum_{\mu\in M_0}G_\mu(\Lambda)\tilde{\pi}^{-w(\mu)}X^{-\mu},
\end{align*}
where $G_\mu(\Lambda) = \sum_{\nu\in({\bf Z}_{\geq 0})^{|A|}} G^{(\mu)}_\nu\Lambda^\nu$ with
\[ G^{(\mu)}_\nu = \begin{cases} \tilde{\pi}^{w(\mu)}\prod_{a\in A}\frac{\pi^{\nu_a}}{\nu_a!} & \text{if $\sum_{a\in A}\nu_a a = -\mu$,} \\
0 & \text{if $\sum_{a\in A}\nu_a a\neq -\mu$.}
\end{cases} \]
Since ${\rm ord}\:\pi^i/i!>0$ for all $i>0$, it follows that $G_\mu(\Lambda)\in R$, $|G_\mu(\Lambda)|\leq |\tilde{\pi}|^{w(\mu)}$, and $G_0(\Lambda)$ is invertible in $R$.  This implies that $\zeta(\Lambda,X)/G_0(\Lambda)\in T$.  
Note also that since $F(\Lambda,X) = \exp(\pi f_{\Lambda}(X))/\exp(\pi f_{\Lambda^p}(X^p))$, it is straightforward to check that
\[ \gamma'(F(\Lambda,X)) = \gamma(F(\Lambda,X)) = \gamma\bigg( \frac{\exp \pi f_{\Lambda}(X)}{\exp\pi f_{\Lambda^p}(X^p)}\bigg) = \frac{\zeta(\Lambda,X)}{\zeta(\Lambda^p,X^p)}. \]
It follows that if $\xi(\Lambda,X)$ is a series satisfying $\gamma(\xi(\Lambda,X))\in S$, then
\begin{align}
\alpha^*(\gamma(\xi(\Lambda,X))) &= \gamma'(F(\Lambda,X)\gamma(\xi(\Lambda^p,X^p))) = \gamma(F(\Lambda,X))\gamma(\xi(\Lambda^p,X^p)) \\  &=\frac{\zeta(\Lambda,X)\gamma(\xi(\Lambda^p,X^p))}{\zeta(\Lambda^p,X^p)}. \nonumber
\end{align}

{\it Remark}:  In terms of the $A$-hypergeometric functions $\{F_i(\Lambda)\}_{i\in M}$ defined in Equation~(1.1), we have $\exp(\pi f_\Lambda(X)) = \sum_{i\in M}F_i(\pi\Lambda)X^i$, so for $i\in M_0$ we have the relation
\begin{equation}
F_i(\pi\Lambda) = \tilde{\pi}^{-w(-i)}G_{-i}(\Lambda).
\end{equation}

\begin{proposition}
The unique fixed point of $\beta$ is $\zeta(\Lambda,X)/G_0(\Lambda)$.
\end{proposition}

\begin{proof}
By Equation~(2.13), we have
\begin{equation}
\alpha^*\bigg(\frac{\zeta(\Lambda,X)}{G_0(\Lambda)}\bigg) = \frac{G_0(\Lambda)}{G_0(\Lambda^p)}\frac{\zeta(\Lambda,X)}{G_0(\Lambda)},
\end{equation}
which is equivalent to the assertion of the proposition.
\end{proof}

By the Remark following Proposition 2.11, $\zeta(\Lambda,X)/G_0(\Lambda)\in T'$.  This gives the following result.
\begin{corollary}
For all $\mu\in M_0$, $G_\mu(\Lambda)/G_0(\Lambda)\in R'$.
\end{corollary}

In the notation of the Remark following Proposition 2.11, one has $\xi(\Lambda,X) = \zeta(\Lambda,X)/G_0(\Lambda)$ and $\eta_0(\Lambda) = G_0(\Lambda)/G_0(\Lambda^p)$, so Equation~(2.12) implies the following result. 
\begin{corollary}
For all $\mu\in M_0$, $G_\mu(\Lambda)/G_0(\Lambda^p)\in R'$.
\end{corollary}

In view of Equation~(2.14), Corollary~2.18 implies that the function ${\mathcal F}(\Lambda) = F_0(\pi\Lambda)/F_0(\pi\Lambda^p)$ converges on the closed unit polydisk, which was the first assertion of Theorem~1.3. 

\section{ $p$-adic Theory}

Fix $\bar{\lambda} = (\bar{\lambda}_a)_{a\in A}\in(\bar{\bf F}_q)^{|A|}$ and let $\lambda = (\lambda_a)_{a\in A}\in(\bar{\bf Q}_p)^{|A|}$, where $\lambda_a$ is the Teichm\"uller lifting of $\bar{\lambda}_a$. We recall Dwork's description of $L(f_{\bar{\lambda}};T)$. Let $\Omega_0 = {\bf Q}_p(\lambda,\zeta_p,\tilde{\pi})$ ($ = {\bf Q}_p(\lambda,\pi,\tilde{\pi})$) and let ${\mathcal O}_0$ be the ring of integers of $\Omega_0$.  

We consider certain spaces of functions with support in $M$. 
We will assume that $\Omega_0$ has been extended by a finite totally ramified
extension so that there is an element $\tilde {\pi}_0 $ in $\Omega_0$ satisfying $\tilde{\pi}_0^D = \tilde {\pi}$. We shall write $\tilde {\pi}^{w(\nu)}$ and mean by it $\tilde {\pi}_0^{Dw(\nu)}$ for $\nu \in M$. Using this convention to
simplify notation, we define
\begin{equation}
B=\bigg\{\sum_{\nu \in M}A_\nu\tilde{\pi}^{w(\nu)}X^{\nu} \mid A_\nu \in
\Omega_0,\;A_\nu\rightarrow 0 \text{ as } \nu \rightarrow \infty \bigg\}.
\end{equation}
Then $B$ is an $\Omega_0$-algebra which is complete under the norm 
$$ \bigg|\sum_{\nu \in M}A_\nu\tilde{\pi}^{w(\nu)}X^{\nu}\bigg| = \sup_{\nu \in M}|A_\nu|. $$

We construct a Frobenius map with arithmetic import in the usual way. Let 
$$   F(\lambda,X) = \prod_{a\in A}\theta(\lambda_aX^a)
                     = \sum_{\mu \in M} B_\mu(\lambda)X^{\mu},     $$
i.e., $F(\lambda,X)$ is the specialization of $F(\Lambda,X)$ at $\Lambda = \lambda$, which is permissible by Lemma~2.2.  
Note also that Lemma~2.2 implies
\[ {\rm ord}\:B_\mu(\lambda)\geq \frac{w(\mu)(p-1)}{p^2}, \]
so we may write $B_\mu(\lambda) = \tilde{\pi}^{w(\mu)}\tilde{B}_\mu(\lambda)$ with $\tilde{B}_\mu(\lambda)$ $p$-integral.

Let 
$$     \Psi(X^{\mu}) = \begin{cases} X^{\mu/p} &  \text{if $p|\mu_i$
    for all $i$,} \\ 0 & \text{otherwise}.
             \end{cases} $$
We show that $\Psi \circ F(\lambda,X) $ acts on $B$.  If 
$\xi = \sum_{\nu \in M}A_\nu\tilde{\pi}^{w(\nu)}X^{\nu} \in B$, then   
$$\Psi\bigg(\bigg(\sum_{\nu\in M} \tilde{\pi}^{w(\mu)} \tilde{B}_\mu(\lambda)X^{\mu}\bigg) \bigg(\sum_{\nu\in M}
A_\nu\tilde {\pi}^{w(\nu)} X^{\nu}\bigg)\bigg) = \sum_{\omega\in M} C_\omega(\lambda) \tilde{\pi}^{w(\omega)}X^{\omega}$$  
where
$$ C_\omega(\lambda) = \sum_\nu \tilde {\pi}^{w(p\omega-\nu)+w(\nu) - w(\omega)}\tilde{B}_{p\omega -\nu}(\lambda)A_\nu $$
(a finite sum).  We have
$$      pw(\omega) = w(p\omega) \leq w(p\omega - \nu) + w(\nu) $$ 
so that 
\begin{equation}
{\rm ord}\:C_\omega(\lambda) \geq \inf_\nu \{{\rm ord}\:\tilde{\pi}^{(p-1)w(\omega)}A_\nu\} = \frac{(p-1)^2w(\omega)}{p^2} + \inf_\nu \{{\rm ord}\:A_\nu\}.
\end{equation}
This implies that $\Psi(F(\lambda,X)\xi) \in B$. 

Let $d(\bar{\lambda}) = [{\bf F}_q(\bar{\lambda}):{\bf F}_q]$, so that $\lambda^{p^{\epsilon d(\bar{\lambda})}}=\lambda$.  Put
\[ \alpha_\lambda = \Psi^{\epsilon d(\bar{\lambda})}\circ\bigg(\prod_{i=0}^{\epsilon d(\bar{\lambda}) - 1} F(\lambda^{p^i},X^{p^i})\bigg). \]
For any power series $P(T)$ in the variable $T$ with constant term $1$, define $P(T)^{\delta_{\bar{\lambda}}} = P(T)/P(p^{\epsilon d(\bar{\lambda})}T)$.
Then $\alpha_\lambda$ is a completely continuous operator on $B$ and the Dwork Trace Formula (see Dwork\cite{D}, Serre\cite{S}) gives
\begin{equation}
L(f_{\bar{\lambda}},\Theta_{\bar{\lambda}},{\bf T}^n;T)^{(-1)^{n+1}} = \det(I-T\alpha_\lambda | B)^{\delta_{\bar{\lambda}}^n}.
\end{equation}

By Equation~(3.2), the $(\omega,\nu)$-entry of the matrix of $\alpha_\lambda$ (\cite[Section 2]{S}) has  ${\rm ord}>0$ unless $\omega=\nu=0$.  The formula for $\det(I-T\alpha_\lambda)$ (\cite[Proposition~7a)]{S}) then shows that this Fredholm determinant can have at most a single unit root.  Since $L(f_{\bar{\lambda}};T)$ has at least one unit root (Section~1), Equation~(3.3) proves that $L(f_{\bar{\lambda}};T)$ has exactly one unit root.

\section{Dual theory}

It will be important to consider the trace formula in the dual theory as well. The basis for this construction goes back to \cite{D+} and \cite{S}.  We define 
$$ B^{\ast} = \bigg\{\xi^*=\sum_{\mu \in M}A^{\ast}_\mu \tilde
{\pi}^{-w(\mu)}X^{-\mu}\mid \text{ $\{A^{\ast}_\mu\}_{\mu\in M}$ is a
bounded subset of $\Omega_0$}\bigg\}, $$
a $p$-adic Banach space with the norm 
    $|\xi^{\ast}| = \sup_{\mu \in M} \{|A^{\ast}_\mu|\}$. 
We define a pairing $\langle\;,\;\rangle: B^*\times B\rightarrow{\Omega}_0$: if $\xi = \sum_{\mu \in M} A_\mu \tilde {\pi}^{w(\mu)}X^{\mu}$, $\xi^{\ast}= \sum_{\mu \in M} A^{\ast}_\mu \tilde{\pi}^{-w(\mu)}X^{-\mu} $, set
\[ \langle\xi^*, \xi\rangle = \sum_{\mu \in M}A_\mu A^{\ast}_\mu \in  {\Omega}_0. \]
The series on the right converges since $A_\mu \rightarrow 0$ as $\mu \rightarrow\infty $ and $\{A^{\ast}_\mu\}_{\mu\in M}$ is bounded.  This pairing identifies $B^*$ with the dual space of $B$, i.e., the space of continuous linear mappings from $B$ to $\Omega_0$ (see \cite[Proposition~3]{S}).  

Let $\Phi$ be the endomorphism of the space of formal series defined by 
\[ \Phi\bigg(\sum_{\mu\in{\bf Z}^n}c_\mu X^{-\mu}\bigg) = \sum_{\mu\in{\bf Z}^n} c_\mu X^{-p\mu}, \]
and let $\gamma'$ be the endomorphism
\[ \gamma'\bigg(\sum_{\mu\in{\bf Z}^n}c_\mu X^{-\mu}\bigg) = \sum_{\mu\in M} c_\mu X^{-\mu}. \]
Consider the formal composition $\alpha_\lambda^* = \gamma'\circ\bigg(\prod_{i=0}^{\epsilon d(\bar{\lambda})-1} F(\lambda^{p^i},X^{p^i})\bigg)\circ\Phi^{\epsilon d(\bar{\lambda})}$.  

\begin{proposition}
The operator $\alpha_\lambda^*$ is an endomorphism of $B^*$ which is adjoint to $\alpha_\lambda:B\to B$.
\end{proposition}

\begin{proof}
As $\alpha_\lambda^*$ is the composition of the operators $\gamma'\circ F(\lambda^{p^i},X)\circ\Phi$ and $\alpha_\lambda$ is the composition of the operators $\Psi\circ F(\lambda^{p^i},X)$, $i=0,\dots,\epsilon d(\bar{\lambda})-1$, 
it suffices to check that $\gamma'\circ F(\lambda,X)\circ\Phi$ is an endomorphism of $B^*$ adjoint to $\Psi\circ F(\lambda,X):B\to B$.  Let $\xi^*(X)= \sum_{\mu \in M}A^{\ast}_\mu\tilde{\pi}^{-w(\mu)}X^{-\mu}\in B^*$.  The proof that the product $F(\lambda,X)\xi^*(X^p)$ is well-defined is analogous to the proof of convergence of the series~(2.3).  We have
\[ \gamma'(F(\lambda,X)\xi^*(X^p)) = \sum_{\omega\in M} C_\omega(\lambda)\tilde{\pi}^{-w(\omega)}X^{-\omega}, \]
where
\begin{equation}
C_\omega(\lambda) =  \sum_{\mu-p\nu=-\omega}\tilde{B}_\mu(\lambda)A^*_\nu
\tilde{\pi}^{w(\omega)+w(\mu)-w(\nu)}. 
\end{equation}
Note that   
$$    pw(\nu) = w(p\nu) \leq w(\omega) + w(\mu) $$
since $p\nu = \omega + \mu $. Thus
$$    (p-1)w(\nu) \leq w(\omega) + w(\mu) -w(\nu), $$
which implies that the series on the right-hand side of (4.2) converges and that $|C_\omega(\lambda)|\leq |\xi^*|$ for all $\omega\in M$.  It follows that $\gamma'(F(\lambda,X)\xi^*(X^p))\in B^*$.  It is straightforward to check that 
$\langle\Phi(X^{-\mu}),X^\nu\rangle =\langle X^{-\mu},\Psi(X^\nu)\rangle$
and that 
\[ \langle \gamma'(F(\lambda,X)X^{-\mu}),X^\nu\rangle =
\langle X^{-\mu}, F(\lambda,X)X^\nu\rangle \]
for all $\mu,\nu\in M$, which implies the maps are adjoint.
\end{proof}

By \cite[Proposition~15]{S} we have $\det(I-T\alpha^*_\lambda\mid B^*) = \det(I-T\alpha_\lambda\mid B)$, so Equation~(3.3) implies 
\begin{equation}
L(f_{\bar{\lambda}},\Theta_{\bar{\lambda}},{\bf T}^n;T)^{(-1)^{n+1}} = \det(I-T\alpha^*_\lambda\mid B^*)^{\delta^n_{\bar{\lambda}}}. 
\end{equation}
From Equations~(2.14) and~(2.16), we have
\[ \alpha^*\bigg(\frac{\zeta(\Lambda,X)}{G_0(\Lambda)}\bigg) = {\mathcal F}(\Lambda)\frac{\zeta(\Lambda,X)}{G_0(\Lambda)}. \]
It follows by iteration that for $m\geq 0$,
\begin{equation}
(\alpha^*)^m\bigg(\frac{\zeta(\Lambda,X)}{G_0(\Lambda)}\bigg) = \bigg(\prod_{i=0}^{m-1}{\mathcal F}(\Lambda^{p^i})\bigg)\frac{\zeta(\Lambda,X)}{G_0(\Lambda)}. 
\end{equation}
We have
\[ \frac{\zeta(\Lambda,X)}{G_0(\Lambda)} = \sum_{\mu\in M_0} \frac{G_\mu(\Lambda)}{G_0(\Lambda)}\tilde{\pi}^{-w(\mu)}X^{-\mu}, \]
so by Corollary~2.17 we may evaluate at $\Lambda = \lambda$ to get an element of $B^*$:
\[ \frac{\zeta(\Lambda,X)}{G_0(\Lambda)}\bigg|_{\Lambda = \lambda} = \sum_{\mu\in M_0} \frac{G_\mu(\Lambda)}{G_0(\Lambda)}\bigg|_{\Lambda = \lambda}\tilde{\pi}^{-w(\mu)}X^{-\mu}\in B^*. \]
It is straightforward to check that the specialization of the left-hand side of Equation~(4.4) with $m = \epsilon d(\bar{\lambda})$ at $\Lambda = \lambda$ is exactly
$\alpha^*_\lambda((\zeta(\Lambda,X)/G_0(\Lambda))|_{\Lambda = \lambda})$,
so specializing Equation~(4.4) with $m=\epsilon d(\bar{\lambda})$ at $\Lambda = \lambda$ gives
\begin{equation}
\alpha^*_\lambda\bigg( \frac{\zeta(\Lambda,X)}{G_0(\Lambda)}\bigg|_{\Lambda = \lambda}\bigg) = \bigg(\prod_{i=0}^{\epsilon d(\bar{\lambda})-1}{\mathcal F}(\lambda^{p^i})\bigg)\frac{\zeta(\Lambda,X)}{G_0(\Lambda)}\bigg|_{\Lambda = \lambda}. 
\end{equation}
Equation~(4.5) shows that $\prod_{i=0}^{\epsilon
  d(\bar{\lambda})-1}{\mathcal F}(\lambda^{p^i})$ is a (unit)
eigenvalue of $\alpha^*_\lambda$, hence by Equation~(4.3) it is the
unique unit eigenvalue of $L(f_{\bar{\lambda}};T)$.

\end{document}